\documentclass[a4paper,twoside]{amsart}
\usepackage[T1]{fontenc}
\usepackage[utf8]{inputenc}

\usepackage{hyperref}
\usepackage[slantedGreeks, partialup, noDcommand]{kpfonts} 
\usepackage{amsthm}
\usepackage{graphicx} 
 
\theoremstyle{plain}
\newtheorem{lemma}{Lemma}[section]
\newtheorem{theorem}{Theorem}[section]

\theoremstyle{definition}

\newcommand{\Cbb}{\mathbb{C}} 
\newcommand{\R}{\mathbb{R}} 

\newcommand{\E}{\mathbb{E}}
\newcommand{\T}{\mathbb{T}}
\newcommand{\Dbb}{\mathbb{D}}

\newcommand{\D}{\mathop{}\!\mathrm{d}}	

\newcommand{\mup}{\mathrm{m}}
\newcommand{\iup}{\mathrm{i}}

\title{Sharp Complex Convexity Estimates}  

\author{Alexander Lindenberger}
\address{A. Lindenberger, Institute of Analysis, Johannes Kepler University Linz,\\Altenberger Straße 69, 4040 Linz, Austria}
\email{alexander.lindenberger@jku.at}
\thanks{The research for this paper was conducted while the first named author was preparing his Master's thesis at the Department of Analysis, Johannes Kepler University Linz. The first and the second named authors were supported by the Austrian Science Foundation (FWF) Pr.Nr P28352-N32}

\author{Paul F. X. Müller}
\address{P. F. X. Müller, Institute of Analysis, Johannes Kepler University Linz,\\Altenberger Straße 69, 4040 Linz, Austria}
\email{paul.mueller@jku.at}

\author{Michael~Schmuckenschläger}
\address{M. Schmuckenschläger, Institute of Analysis, Johannes Kepler University Linz,\\Altenberger Straße 69, 4040 Linz, Austria}
\email{michael.schmuckenschlaeger@jku.at}

\keywords{martingale inequalities, complex convexity, best constants}
\subjclass[2010]{Primary 60G42; Secondary 32U05}
 
\begin{document}

\begin{abstract}
In this paper we determine the value of the best constants in the $2$\hbox{-}uniform $PL$-convexity estimates of $\Cbb$. This solves a problem posed by W. J. Davis, D. J. H. Garling and N. Tomczak-Jaegermann \cite{davis1984}.
\end{abstract}

\maketitle

\section{Introduction}

Quoting from the book of Pisier \cite[section 11.9]{pisier2016}, ``Haagerup was the first who noticed the importance of the (sharp) inequality''
\begin{equation} \label{inequality_Pisier}
\left( |x|^2 + \frac 1 2 |y|^2 \right)^{\frac 1 2} \le \int_\T |x+ \zeta y|\D\mup(\zeta) \qquad \text{for } x,y\in\Cbb,
\end{equation}
where $\T$ denotes the complex unit circle and $\mup$ the usual Haar measure on $\T$ with $\mup(\T)=1$.
W. J. Davis, D. J. H. Garling and N. Tomczak-Jaegermann  \cite[Proposition 3.1]{davis1984} present a proof of \eqref{inequality_Pisier} based on the power series representation of elliptic integrals. U. Haagerup and G. Pisier obtained an extension of \ref{inequality_Pisier} to non-commutative $H^1$-spaces in \cite{haagerup_pisier}.

W. J. Davis, D. J. H. Garling and N. Tomczak-Jaegermann  \cite[Problem 4]{davis1984} conjecture, that, for $\alpha\in(0,2]$, the inequality
\begin{equation} \label{inequality_C}
 \left(|x|^2 + \frac \alpha 2 |y|^2 \right)^{\frac 1 2}  \le \left( \int_\T |x+\zeta y|^\alpha \D \mup(\zeta) \right)^{\frac 1 \alpha}  \qquad \text{for } x,y\in\Cbb
\end{equation}
holds true and that the constant $\frac \alpha 2$ is sharp. In this paper we give a proof of this conjecture.

In \cite{davis1984} the notion of \emph{uniform $PL$-convexity} is introduced for continuously quasi-normed vector spaces over the complex numbers.
For $r\in [2,\infty)$, such a space $(E,\|\cdot\|)$ is called \emph{$r$-uniformly $PL$-convex}, if and only if for any $\alpha\in(0,\infty)$ there exists a $\lambda>0$, such that
\begin{equation}\label{inequality_quasi}
 \left(\|x\|^r + \lambda\|y\|^r \right)^{\frac 1 r} \le \left(\int_\T \|x + \zeta y \|^\alpha \D \mup(\zeta) \right)^{\frac 1 \alpha} \qquad \text{for } x,y\in E.
\end{equation}
The largest possible value of $\lambda$ satisfying \eqref{inequality_quasi} is then denoted by $I_{r,\alpha}(E)$. Using this notion, $\Cbb$ itself is $2$\hbox{-}uniformly $PL$\hbox{-}convex and it is easily seen, that $I_{2,\alpha}(\Cbb) =1$ for $\alpha\ge 2$. We thus identify the value of $I_{2,\alpha}(\Cbb)$ for each $\alpha\in (0,\infty)$ by proving the following theorem:

\begin{theorem} \label{theorem_main}
 Let $\alpha\in(0,2]$ and $\lambda \le \frac \alpha 2$. We have 
 \begin{equation} \label{inequality_theorem}
 (|x|+ \lambda |y|^2)^{\frac 1 2} \le \left( \int_{\T} |x+ \zeta y|^\alpha \D \mup(\zeta) \right)^{\frac 1 \alpha} \qquad \text{for } x,y\in \Cbb,
 \end{equation}
 and $\frac \alpha 2$ is the best (i. e. largest) real constant satisfying \eqref{inequality_theorem}.
\end{theorem} 

\section{Preliminaries} \label{sec2}

In this section we collect three main ingredients employed in the proof of Theorem \ref{theorem_main}.

The proof of $I_{2,1} = \frac 1 2$ given in \cite{davis1984} relies on the power series representation of elliptic integrals and seems difficult to generalize. In 1994 the third named author gave an alternative proof, using techniques of stochastic analysis. The main idea was rewriting the line integral on the right-hand side of \eqref{inequality_Pisier} as an area integral. For general $\alpha$ this reads as follows.

\begin{lemma} \label{lemma_schmuck}
 For $\alpha\in(0,\infty)$ and $y>0$, we have
 $$
  \int_{\T}|1+y \zeta|^\alpha \D \mup(\zeta) = 1+ \frac {\alpha^2 y^2} 2 \int_\Dbb |1+y z|^{\alpha-2} G(z) \D z,
 $$
 where $\Dbb$ denotes the complex unit disk and $G$ Green's function of this domain and pole $0$, so $G(z) = \frac 1 \piup \ln \frac 1 {|z|}$ for each $z\in \Dbb$.
\end{lemma}

Next we recall the inequality of arithmetic and geometric means in the form used below. For details we refer to \cite[section 6.7]{hardy1934inequalities}.

\begin{theorem}[AM-GM inequality]
 Let $f\colon \T \to \overline \R$ be non-negative almost everywhere on $\T$ and measurable. Let $r$ be a positive real number. Then we have
 \begin{equation} \label{AM-GM-ineq}
  \left(\int_\T f^{-r} \D \mup \right)^{-\frac 1 r} \le \exp \int_\T \ln f  \D \mup \le \left(\int_\T f^r \D \mup \right)^{\frac 1 r},
 \end{equation}
 if all these integrals exist.
\end{theorem}

The usefulness of the AM-GM inequality is due to the fact, that for our specific functions, the geometric mean is easier to calculate than the arithmetic means. Indeed we can evaluate the geometric mean using the familiar Jensen formula (see \cite[p. 207-208]{ahlfors1966complex}). For integrals over the complex unit circle, it has the following form.

\begin{theorem}[Jensen's formula] \label{Jensen}
 Let $f\colon U \to \Cbb$ be analytic on a domain $U \subseteq \Cbb$, where $U$ contains the closed unit disk $\overline \Dbb$. Let $a_1, a_2, \ldots, a_n$ be the zeros of $f$ in the interior of $\overline \Dbb$ (repeated according to multiplicity). Assuming $f(0)\ne 0$, we have
 $$
  \int_\T\ln |f(\zeta)| \D \mup(\zeta) = \ln |f(0)| + \sum_{i=1}^n \ln \frac 1 {|a_i|} .
 $$
\end{theorem}

\section{Proof of Theorem \ref{theorem_main}} \label{sec_proof}

Using the tools collected in section \ref{sec2}, we can now give a proof of Theorem \ref{theorem_main}. Note that for $x=0$ or $y=0$, the inequality \eqref{inequality_theorem} is trivial. Hence, by rescaling and rotation, it suffices to show, that
\begin{align}\label{inequality_R}
 \left(\int_\T |1+y\zeta|^\alpha\D \mup(\zeta)\right)^{\frac 1\alpha} \ge \left(1+\lambda y^2\right)^{\frac 1 2} \qquad \text{for each } y>0
\end{align}
holds for $\lambda=\frac\alpha 2$ and does not hold for $\lambda>\frac \alpha 2$.

We first show, that $\lambda= \frac \alpha 2$ satisfies \eqref{inequality_R}.
To this end, we define  functions $a,h,g\colon [0,\infty) \to \R$ by
\begin{align*}
 a(y) &:= \int_\T |1+y\zeta|^\alpha \D \mup (\zeta),\\
 h(y) &:= \begin{cases}
        1+\frac{\alpha^2}{4} y^2,  &y^2\in [0,1]
        \\
        \frac{\alpha^2} 4 + \left(y^2\right)^{\frac \alpha 2} - \frac \alpha 2 \left( 1-\frac \alpha 2 \right) \ln y^2, &y^2\in \left(1, \left(1-\frac \alpha 2\right)^{-1} \right) \\
        y^\alpha,  &y^2 \ge\left(1-\frac \alpha 2 \right)^{-1}
    \end{cases}, \\
 g(y) &:= \left(1+\frac \alpha 2 y^2 \right)^{\frac \alpha 2}
\end{align*}
and show
\begin{equation} \label{inequality_proof}
 a(y) \ge h(y) \ge g(y) \qquad \text{for } y > 0.
\end{equation}
\begin{figure}[htp]
    \centering
    \includegraphics[width=.3\textwidth]{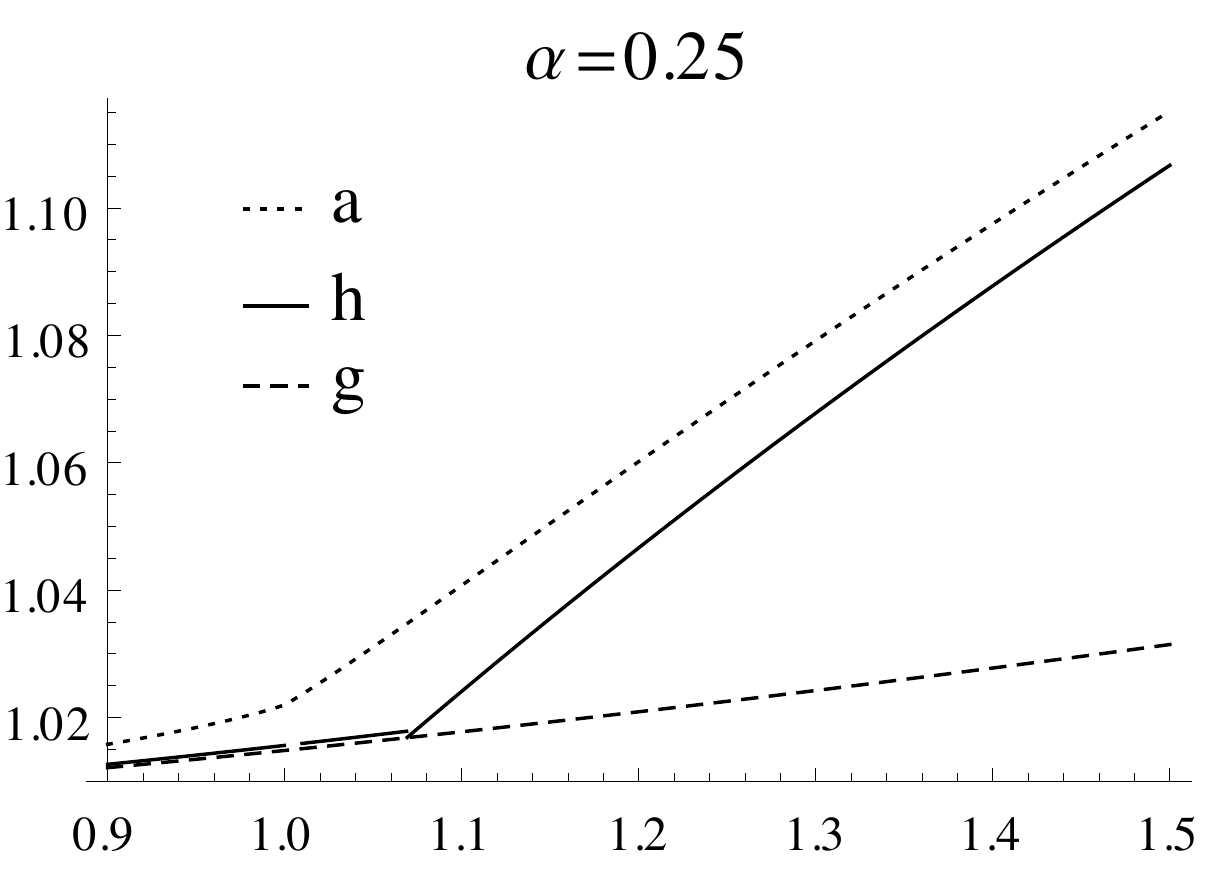}\hfill 
    \includegraphics[width=.3\textwidth]{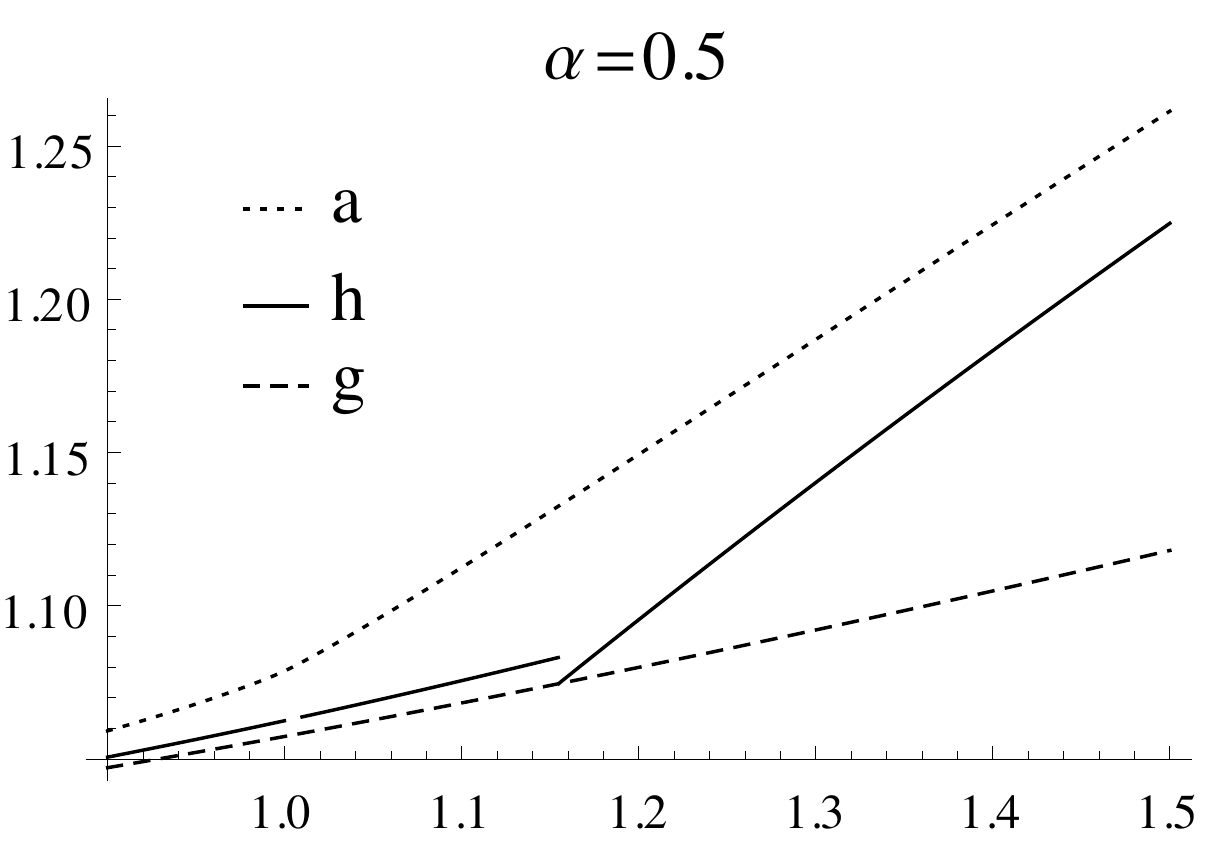}\hfill
    \includegraphics[width=.3\textwidth]{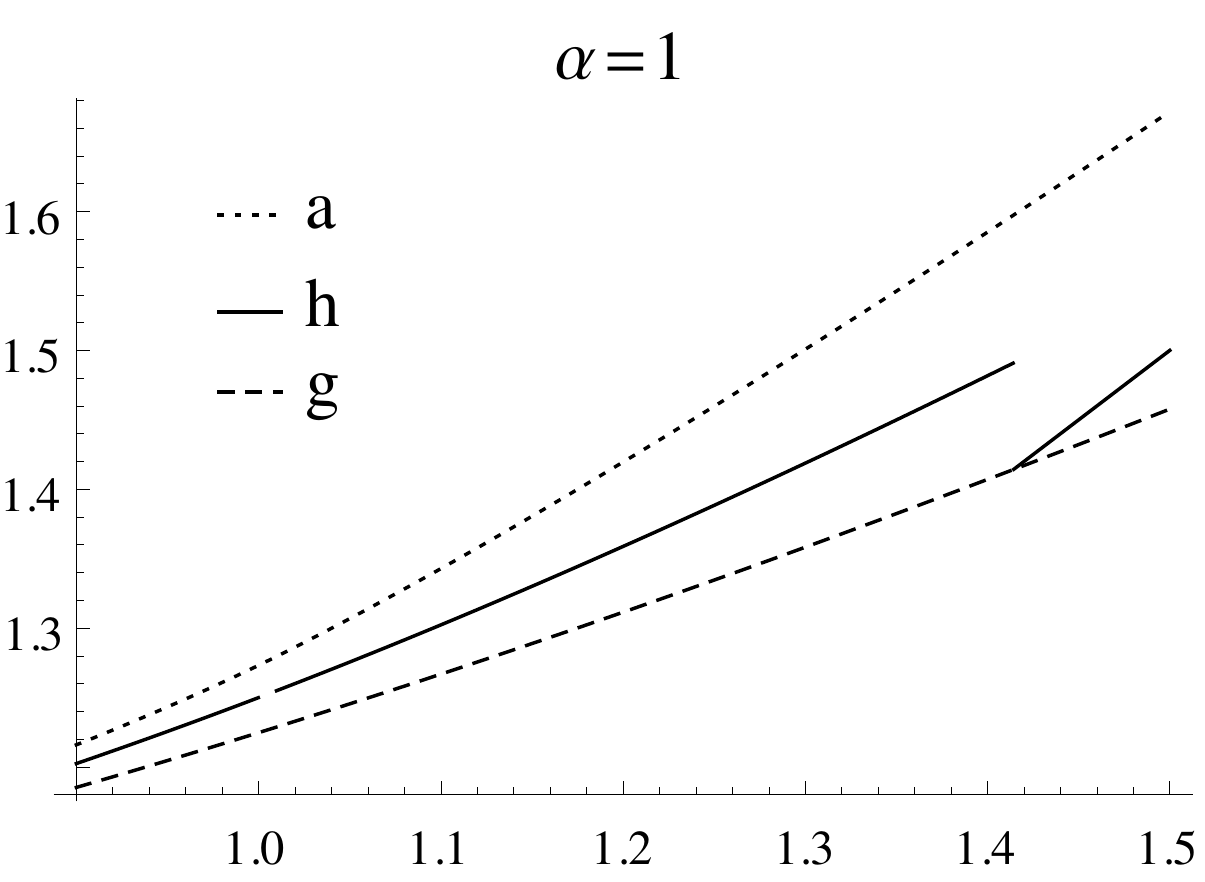}
    \caption{Plots of the functions $a, h, g$ with different values of $\alpha$} 
    \label{figure2}
\end{figure} 

Therefore, we fix $y>0$ and distinguish three cases, according to the definition of $h$.
 \begin{description}
   \item[Case $y^2 \ge \left(1-\frac \alpha 2\right)^{-1}$:]
    Using the AM-GM inequality and applying Jensen's formula to the analytic function $z\mapsto 1+ y z$ yield
    $$
     \left( \int_\T |1+y\zeta|^\alpha\D \mup(\zeta) \right)^{\frac 1 \alpha} \ge \exp \int_\T \ln |1+y \zeta| \D \mup(\zeta) = \exp(\max\{0,\ln y\}) = \max\{1,y\}.
    $$
    In this case, we particularly have $y\ge 1$, so we conclude
    $$
     a(y) \ge y^\alpha = h(y).
    $$
    On the other hand, $y^2 \ge \left(1-\frac \alpha 2\right)^{-1}$ implies
    $$
     h(y) = y^\alpha \ge \left(1+\frac \alpha 2 y^2\right)^{\frac \alpha 2} = g(y).
    $$
 \end{description}
For the other two cases, we use the alternative representation of $a$ obtained by Lemma \ref{lemma_schmuck}, i. e. 
\begin{align*}
 a(y) &= 1 + \frac{\alpha^2 y^2} 2 \int_{\Dbb} |1+y z|^{\alpha-2}G(z) \D z \\
 &=1 + \alpha^2 y^2 \int_0^1 \int_\T |1+y r \zeta|^{\alpha-2} \D\mup(\zeta) r \ln \frac 1 r \D r.
\end{align*}
We note $\alpha-2\le 0$, so the AM-GM inequality for negative exponents and Jensen's formula, applied to the analytic function $z\mapsto 1+y r z$, yield
  $$
    \left( \int_{\T} |1+ y r \zeta |^{\alpha-2} \D \mup(\zeta) \right)^{\frac 1 {\alpha-2}} \le \exp \int_\T \ln |1+ y r \zeta| \D \mup(\zeta) = \max\{ 1,y r\}.
  $$
Hence, we arrive at
$$
    a(y) \ge 1+ \alpha^2 y^2 \int_0^1 \min\left\{1,(yr)^{\alpha-2} \right\} r\ln \frac 1 r \D r.
$$

 \begin{description}
    \item[Case $y^2 \in\left(1, \left(1-\frac \alpha 2\right)^{-1} \right)$:]
    By calculation, we obtain 
    \begin{align*}
        \int_0^1 \min \left\{ 1,(yr)^{\alpha-2} \right\} r \ln \frac 1 r \D r &= \int_0^{\frac 1 y} r \ln \frac 1 r \D r + y^{\alpha-2} \int_{\frac 1 y}^1 r^{\alpha-1} \ln \frac 1 r \D r \\
        &= \frac {1+ 2 \ln y}{4 y^2} + y^{\alpha-2} \left( \frac 1 {\alpha^2} - \frac {1+\alpha\ln y}{\alpha^2 y^\alpha} \right).
    \end{align*}
    Therefore, we have
    $$
        a(y) \ge \frac{\alpha^2} 4 + \left(y^2\right)^{\frac \alpha 2} - \frac \alpha 2 \left( 1-\frac \alpha 2 \right) \ln y^2 = h(y).
    $$
    The inequality $h(y) \ge g(y)$ is an exercise in elementary calculus, carried through in the appendix (Lemma \ref{lemma3} with $\beta := \frac \alpha 2$ and $t:= y^2$).
    \item[Case $y\le 1$:]
    Here, the integral simplifies to
    $$
        \int_0^1 \min\left\{ 1,(y r)^{\alpha-2} \right\} r \ln \frac 1 r \D r = \int_0^1 r \ln \frac 1 r \D r = \frac 1 4.
    $$
    Consequently, we have
    $$
        a(y) \ge 1+\frac{\alpha^2}{4} y^2 = h(y).
    $$
    The inequality $h(y)\ge g(y)$ is again verified by elementary calculus (Lemma \ref{lemma2} with $\beta := \frac \alpha 2$ and $t:= y^2$).
 \end{description}
 
 It remains to show, that $\frac \alpha 2$ is the largest constant satisfying \eqref{inequality_R}. To this end, we fix $\lambda \in \R$, define $b(y) := (1+ \lambda y^2)^{\frac \alpha 2}$ for each $y \ge 0$, and assume
 $$
  a(y) \ge b(y) \quad \text{for } y>0.
 $$
 Obviously, we have $a(0)=b(0)=1$ and a short calculation shows $a'(0)=b'(0) = 0$. From the assumption, we now conclude $a''(0) \ge b''(0)$. Using $a''(0) = \frac{\alpha^2} 2$ and $b''(0) = \alpha \lambda$, we obtain $\lambda \le \frac \alpha 2$. 
 \qed

\section{Appendix} \label{secAppendix}

\subsection{Proof of Lemma \ref{lemma_schmuck}}

We give a proof of the identity obtained and exploited by the third named author, i. e.
\begin{equation} \label{formula_schmuck}
 \int_\T |1+y\zeta|^\alpha \D \mup(\zeta) = 1+ \frac{\alpha^2 y^2} 2 \int_\Dbb |1+yz|^{\alpha-2} G(z) \D z \qquad \text{for } y>0, \alpha \in (0,\infty).
\end{equation}
Here we use techniques of stochastic analysis. Alternatively, one could obtain \eqref{formula_schmuck} using Gauss's divergence theorem. See for instance \cite[p. 355-356]{mueller05}.

Let $(B_t)_{t\ge 0} = \left(B_t^{(1)} + \iup B_t^{(2)}\right)_{t\ge 0}$ be a complex Brownian motion started at the origin and $\tau := \inf\left\{ t>0 : |B_t|>1 \right\}$. We then have
 $$ 
  \int_{\T} |1+ y \zeta |^\alpha \D \mup(\zeta)= \E\left[ |1+y B_\tau|^\alpha \right].
 $$
 Applying It\^o's formula to the continuous martingale $(1+y B_t)_{t>0}$ gives
    \begin{align*}
    |1+y B_t|^\alpha = 1 &+ \alpha y \int_0^t \left(1+y B_s^{(1)}\right) |1+y B_s|^{\alpha-2} \D B_s^{(1)} + \alpha y^2 \int_0^t B_s^{(2)} |1+y B_s|^{\alpha-2} \D B_s^{(2)} \\
    &+\frac {\alpha^2} 2 y^2 \int_0^t |1+y B_s|^{\alpha-2} \D s,
    \end{align*}
    whenever $t < \tau$, and hence
    $$
     \E\left[ |1+y B_\tau|^\alpha \right] = 1 + \frac{\alpha^2} 2 y^2 \E\left[ \int_0^\tau |1+y B_s |^{\alpha-2} \D s\right].
    $$ 
    Let $G$ denote Green's function of the unit disk with pole $0$.
    Using the occupation time formula for the Brownian motion \cite[Theorem 3.32, section 3.3, p. 80]{moerters2010brownian}, we obtain 
    \begin{align*}
    1+ \frac {\alpha^2} 2 y^2 \E\left[\int_0^\tau |1+y B_s |^{\alpha-2} \D s  \right] = 1+ \frac {\alpha^2} 2 y^2 \int_\Dbb |1+y z|^{\alpha-2} G(z) \D z.
    \end{align*} 
    \qed

\subsection{Elementary estimates}

Here we collect and prove the elementary estimates used in section \ref{sec_proof}. We begin with a special case of the classical Bernoulli inequality for real exponents.

\begin{lemma}[Bernoulli's inequality] \label{lemma2}
 Let $\beta \in[0,1]$ and $t>0$. We have
 $$
  1+\beta^2 t \ge (1+\beta t)^\beta.
 $$
\end{lemma}
\begin{proof}
 Taylor's theorem, applied to $\phi(t) := (1+\beta t)^\beta$, guarantees the existence of a $\xi \in [0,t]$, such that
 $$
  \phi(t) = \phi(0) + \phi'(0) t + \phi''(\xi) \frac{t^2} 2 = 1+ \beta^2 t + \beta^3 (\beta-1) (1+\beta\xi)^{\beta-2} \frac{t^2} 2.
 $$
 Using this $\xi$, we obtain
 $$
  1+\beta^2 t - (1+\beta t)^\beta = \beta^3 (1-\beta) (1+\beta \xi)^{\beta-2} \frac{t^2} 2 \ge 0.
 $$
\end{proof}

\begin{lemma} \label{lemma3}
 Let $\beta \in[0,1]$ and $t\in\left[1,\frac 1 {1-\beta}\right]$. We have
 $$
  \beta^2+t^\beta - \beta(1-\beta)\ln t \ge (1+\beta t)^\beta.
 $$
\end{lemma}
\begin{proof}
 Taylor's theorem, applied to $\phi(t) := \beta^2+t^\beta-\beta(1-\beta)\ln t -(1+\beta t)^\beta$, guarantees the existence of a $\xi \in [1,t]$, such that
 $$
  \phi(t) = \phi(1)+(t-1)\phi'(\xi) = \beta^2+1-(1+\beta)^\beta + (t-1) \beta \left( \xi^{\beta-1} - (1-\beta)\xi^{-1} - \beta(1+\beta\xi)^{\beta-1} \right).
 $$ 
 From Lemma \ref{lemma2} we know
 $$
  \beta^2 + 1 -(1+\beta)^\beta \ge 0.
 $$
 $\xi\in\left[1,\frac 1 {1-\beta} \right]$ implies $\xi\le 1+\beta \xi$ and thus
 $$
  \xi^{\beta-1} - (1-\beta)\xi^{-1} - \beta(1+\beta\xi)^{\beta-1} \ge \xi^{\beta-1} - (1-\beta)\xi^{-1} -\beta \xi^{\beta-1} = (1-\beta) \left( \xi^{\beta-1} - \xi^{-1} \right) \ge 0.
 $$
\end{proof}

\bibliographystyle{abbrv} 
\bibliography{references}

\end{document}